\documentclass[preprint,,authoryear,11pt,a4paper]{elsarticle}

\usepackage{amsthm,amssymb,amsfonts,amsmath}

\usepackage{graphicx,bbm,color}
\usepackage{hyperref}

\newtheorem{theo}{Theorem}[section]
\newtheorem{prop}{Proposition}[section]
\newtheorem{rem}{Remark}[section]

\newtheorem{cor}{Corollary}[section]

\newcommand{\rme}{{\rm e}}
\newcommand{\rmd}{{\rm d}}
\newcommand{\la}{{\lambda}}
\newcommand{\ep}{{\varepsilon}}
\newcommand{\1}{{\bf 1}}

\newcommand{\F}{{\mathcal F}}
\newcommand{\PP}{{\mathbb P}}
\newcommand{\QQ}{{\mathbb Q}}
\newcommand{\RR}{{\mathbb R}}
\newcommand{\ZZ}{{\mathbb Z}}
\newcommand{\EE}{{\mathbb E}}

\begin{document}

\begin{frontmatter}
\title{Damped jump-telegraph processes}
\author{Nikita Ratanov}
\ead{nratanov@urosario.edu.co}

\address{Universidad del Rosario, Cl.\,12c, No.~4-69, Bogot\'a, Colombia}



\begin{abstract} 
We study a one-dimensional  Markov modulated random walk with jumps.
It is assumed that  amplitudes of jumps 
as well as a chosen velocity regime are random and
depend on a  time spent by the process at a previous state of the underlying Markov process.

Equations for the distribution and equations for its moments  are derived. 
We characterise the martingale distributions in terms of  
observable proportions between 
jump and velocity regimes.
  \end{abstract} 

\begin{keyword} inhomogeneous  jump-telegraph process, Volterra equation, martingale measure

\medskip
\MSC{primary 60J27; secondary 60J75, 60K99}
\end{keyword}

\end{frontmatter}

\section{Introduction}
Telegraph processes with different switchings and velocity regimes 
are studied recently in connection with possibility of different applications
such as, for instance, queuing theory (see   \citet{Zacks1}, \citet{Zacks2}) and
mathematical biology (see \citet{Hadeler-rev}). 
Special attention is devoted to financial applications (see \citet{R2007}, \citet{ON2012}).
In the latter case, an  arbitrage reasoning demands the presence of jumps.

The motions with deterministic jumps are studied in detail, see 
the formal expressions of the transition densities  
in  \citet{R2007},  \citet{DiC2013}. 
Such a model is developed  for the option pricing problem, which is based
 on the risk-neutral approach, see  \citet{R2007}.
If the jump amplitudes are random,  the case is less known.
The telegraph processes of this type are studied 
earlier only under the assumption of mutual independence 
of jump values and jump amplitudes,  
see \citet{Zacks2} and \citet{DiC2013}. Similar setting were used
for the purposes of financial applications, \citet{ON2012}.

We present here  a jump-telegraph process   
when an amplitude of the next jump depends on 
the (random) time spent by the process at the previous state.  
This approach is of special interest for the economical and 
the financial applications, everywhere when the comportment 
of process relates with friction and memory.

Assume that the particle  moves with random
(and \emph{variable})  velocities  performing jumps 
of random amplitude whenever the velocity is changed. 
More precisely,  the actual velocity regime and the amplitude of the next  jump
are defined as (alternated)
 functions of the  time spent by the particle at the previous state.
We assume also that the time intervals between the subsequent state changes 
have sufficiently arbitrary alternated distributions. 
It creates an effect of damping process where a friction is generated by means of memory.

This setting generalises processes which were used before for market modelling by
\cite{R2007} and \cite{ON2012}.

The underlying processes are described in Sections \ref{sec:2}-\ref{sec:2-2}.
Section \ref{sec:mart} presents the result which can be interpreted as a Doob-Meyer decomposition.
Several examples with different  regimes of velocities and of jumps are presented.

\section{Generalised jump-telegraph processes: distribution}\label{sec:2}
\setcounter{equation}{0}
Let  $(\Omega, \mathcal{F}, \PP)$ be a probability space.
Consider two continuous-time Markov processes $\ep_0(t), \ep_1(t)\in\{0, 1\},\; t\in(-\infty, \infty)$.
The subscript $i\in\{0, 1\}$ indicates the initial state, $\ep_i(0)=i$ (with probability 1).
Assume that $\ep_i=\ep_i(t),\; t\in(-\infty, \infty)$ are left-continuos a. s.

Let $\{\tau_n\}_{ n\in\ZZ}$ be a Markov flow of switching 
times. The increments $T_n:=\tau_n-\tau_{n-1}, n\in\ZZ$ are  independent 
and possess alternated distributions   
(with the distribution  functions $F_0$, $F_1$, 
the survival functions $\bar F_0,\; \bar F_1$ and 
the densities $f_0,\; f_1$).  We assume that 
$\tau_0=0$, i. e. the state process $\ep_i$ is started at the switching instant.
The distributions of $\tau_n$ and $T_n$  depend on the initial state $i,\;i\in\{0, 1\}$.
For brevity, we will not always indicate this dependence. 

Consider a particle moving on $\RR$ with two 
alternated velocity regimes $c_0$ and $c_1$.  
These velocities are described by  two   continuous   functions
$c_i=c_i(T, t),\;  T, t>0,\; i=0,1$. 
At each instant $\tau_n$
the particle takes the velocity regime $c_{\ep_i(\tau_{n})}(T_{n}, \cdot)$, where
$T_{n}$ is the (random) time spent by the particle at the previous state. 
 We define 
a pair of the (generalised) telegraph processes 
 $\mathcal T_i,\; i=0, 1$ 
driven by variable  velocities $c_0,\; c_1$ as follows,
\begin{equation}\label{def:tp}\begin{aligned}
\mathcal T_0(t)=\mathcal T_0(t; c_0, c_1)=
&\sum_{n=0}^{\infty}c_{\ep_0(\tau_{n})}(T_{n}, t-\tau_{n})\1_{\{\tau_{n}<t\leq\tau_{n+1}\}},\\
\mathcal T_1(t)=\mathcal T_1(t; c_0, c_1)=
&\sum_{n=0}^{\infty}c_{\ep_1(\tau_{n})}(T_{n}, t-\tau_{n})\1_{\{\tau_{n}<t\leq\tau_{n+1}\}},
\end{aligned}\quad t\geq0.\end{equation}
The integral $\int_0^t\mathcal T_i(s)\rmd s,\; i=0, 1$  is named the \emph{integrated telegraph process}.

Let $N=N_i(t):=\max\{n\geq0~:~\tau_n\leq t\},\; t\geq0$ be a counting process.
 Notice that, $N_i(0)=0$ and 
$\ep_0(t)=(1-(-1)^{N_0(t)})/2$ and $\ep_1(t)=(1+(-1)^{N_1(t)})/2$.

The integrated telegraph process can be 
interpreted as the sum of random number of random variables.
If $N_i(t)=0$, then 
\begin{equation}
\label{def:itp0}
\int\limits_0^t\mathcal T_i(s)\rmd s=l_i(T_0;  t);
\end{equation}
if  $N_i(t)>0$, then the integrated telegraph process is expressed as
\begin{equation}
\label{def:itp}
\int\limits_0^t\mathcal T_i(s)\rmd s=\sum_{n=0}^{N_i(t)-1}l_{\ep_i(\tau_n)}(T_n; \tau_n, \tau_{n+1})
+l_{\ep_i(\tau_{N_i(t)})}(T_{N_i(t)}; \tau_{N_i(t)}, t).
\end{equation}
Here 
\begin{equation*}
\label{def:li}
l_i(T; u, t):=\int_u^tc_i(T, s)\rmd s,\quad i=0, 1.
\end{equation*}
Notice that
$
l_i(T; u, s)+l_i(T; s, t)\equiv l_i(T; u, t),\;  i=0, 1.
$
Simplifying  we  write $l_i(T; t)$ instead of $l_i(T; 0, t)$.



In the same manner we define the jump component.
Let  $h_0=h_0(T)$ and $h_1=h_1(T),\; T\geq0$ 
be  a pair of deterministic continuous (or, at least, boundary measurable) functions. 
Consider telegraph processes \eqref{def:tp} based on $h_i(T)$ instead of $c_i=c_i(T, \cdot),\; i=0,1$,
\[
\mathcal T_i(t; h_0, h_1)=\sum_{n=1}^\infty h_{\ep_i(\tau_{n})}(T_{n})\1_{\{\tau_{n}<t\leq\tau_{n+1}\}},
\quad i=0, 1.
\]

An integrated jump process is defined in
the form of compound Poisson process
by the integral
\begin{equation}
\label{def:ijp}
\int\limits_0^t\mathcal T_i(s; h_0, h_1)\rmd N_i(s)
=\sum_{n=1}^{N_i(t)}h_{\ep_i(\tau_{n})}(T_n),\quad i=0, 1.
\end{equation}
The amplitude of a jump depends on the time spent by the particle at the current state.

Finally, the generalised integrated jump-telegraph process 
is the sum of the integrated telegraph process
defined by \eqref{def:itp0}-\eqref{def:itp} and the jump component defined by \eqref{def:ijp}:
\begin{equation}\label{def:ijtp}
X_i(t)=\int\limits_0^t\mathcal T_i(s; c_0, c_1)\rmd s
+\int\limits_0^t\mathcal T_i(s; h_0, h_1)\rmd N_i(s),\quad t\geq0,\; i=0, 1.
\end{equation}
It describes the particle which moves,   alternating the velocity regimes at random times $\tau_n$,
starting from the origin  at the velocity regime $c_i$.
 Each velocity reversal is accompanied
by jumps of random amplitudes, $X_i(t)$ is the current particle's position.

Conditioning on the first velocity reversal, notice that
\begin{equation}
\label{eq:X01}
\begin{aligned}
X_0(t)\stackrel{D}{=}&l_0(T_{0}; t)\1_{\{\tau_1>t\}}+\left[l_0(T_{0}; \tau_1)
+h_0(\tau_1)+X_1(t-\tau_1)\right]
\1_{\{\tau_1<t\}},\\
X_1(t)\stackrel{D}{=}&l_1(T_{0}; t)\1_{\{\tau_1>t\}}+\left[l_1(T_{0}; \tau_1)
+h_1(\tau_1)+X_0(t-\tau_1)\right]
\1_{\{\tau_1<t\}}.
\end{aligned}\end{equation}
Here $\stackrel{D}{=}$ denotes the equality in distribution.
At each of two equalities the first term represents 
the movement without velocity reversal; the second one is the 
sum of three terms: the path till the first reversal, 
the jump value and the movement which is initiated after the first reversal.

The distribution of $X(t),\;t>0$ is separated into the singular and absolutely continuous parts.

The singular part of the distribution corresponds to the movement without any velocity reversals;
let $\PP^{(0)}_i,\; i=0, 1$ be  the respective  conditional distribution, if the initial state $i=\ep_i(0)$ is fixed:   
for any Borel set $A$ we set
\[
\PP^{(0)}_i(A)
:=\PP(X_i(t)\in A,\; N_i(t)=0),  
\quad i=0, 1.
\]
We denote the corresponding expectation by  $\EE^{(0)}_i\{\cdot\}$. 
On the space of (continuous) test-functions $\varphi$ consider the linear functional (generalised function),
$
\varphi\to\EE^{(0)}_i\{\varphi(X(t))\}.
$
It is easy to see that
\[
\EE^{(0)}_i\{\varphi(X(t))\}=\int\limits_{-\infty}^\infty\varphi(y)\PP^{(0)}_i(\rmd y)
=\bar F_i(t)\int\limits_0^\infty \varphi(l_i(s; t))f_{1-i}(s)\rmd s=:<p_i(\cdot, t; 0),\; \varphi>.
\]
The generalised function 
\begin{equation}
\label{eq:p0}
p_i(x, t; 0)=\bar F_i(t)\int_0^\infty \delta_{l_i(s; t)}(x)f_{1-i}(s)\rmd s=
\bar F_i(t)\int_0^\infty \delta_0(x-l_i(s; t))f_{1-i}(s)\rmd s
\end{equation}
can be viewed as the distribution  ``density''. Here $\delta_a(x)$ 
is the Dirac measure (of unit mass) at point $a$.

The absolutely continuous part of the distribution  of $X_i(t)$ is characterised by the densities
\[
p_i(x, t; n)=\PP\{X_i(t)\in\rmd x,\; N_i(t)=n\}/\rmd x,\quad i=0, 1,\; n\geq1.
\]

The sum
\[
p_i(x, t)=\sum_{n=1}^\infty p_i(x, t; n)
\]
corresponds to the absolutely continuous part of distribution of $X_i(t),\;  i=0, 1$. 

Conditioning on the first velocity reversal, similarly to \eqref{eq:X01}
we obtain the following equations, $n\geq 1,$
\begin{equation}\label{eq:pi}
\begin{aligned}
p_0(x, t; n)=&\int_0^\infty f_1(\tau)\rmd\tau\int_0^t  p_1(x-l_{0}(\tau; s)-h_{0}(s), t-s; n-1)f_0(s)\rmd s,\\
p_1(x, t; n)=&\int_0^\infty f_0(\tau)\rmd \tau\int_0^t p_0(x-l_{1}(\tau, s)-h_{1}(s), t-s; n-1)f_1(s)\rmd s
\end{aligned}\end{equation}
(if $n=1$ the inner integrals are understood in the sense of the theory of generalised functions).
Summing up in \eqref{eq:pi}  we get the system of integral equations for (complete) distribution densities,
\begin{equation}\label{eq:p-complete}
\begin{aligned}
p_0(x, t)=&p_0(x, t; 0)+\int_0^\infty f_1(\tau)\rmd\tau\int_0^t  p_1(x-l_{0}(\tau; s)-h_{0}(s), t-s)f_0(s)\rmd s,\\
p_1(x, t)=&p_1(x, t; 0)+\int_0^\infty f_0(\tau)\rmd \tau\int_0^t p_0(x-l_{1}(\tau, s)-h_{1}(s), t-s)f_1(s)\rmd s.
\end{aligned}\end{equation}
Here $p_0(x, t; 0)$ and $p_1(x, t; 0)$
are defined by \eqref{eq:p0}.

If $c_0, c_1\equiv const,\; h_0, h_1\equiv const$ equations  \eqref{eq:pi}   and 
\eqref{eq:p-complete} can be solved explicitly using the following notations,
\begin{equation*}
\label{def:xi}
\xi=\xi(x, t):=\frac{x-c_1t}{c_0-c_1}\text{~~~and ~~~} t-\xi=\dfrac{c_0t-x}{c_0-c_1}. 
\end{equation*}
Notice that $0<\xi(x, t)<t$, if $x\in(c_1t,\; c_0t)$ (say, $c_0>c_1$). 
Define  the
 functions $q_i(x,t;n)$, $i=0,1$: for $c_1t<x<c_0t$, 
\begin{equation}\label{eq:qipar}
\begin{aligned}
q_0(x, t; 2n)=&\frac{\lambda_0^n\lambda_1^n}{(n-1)!n!}\xi^n(t-\xi)^{n-1} \\
q_1(x, t; 2n)=&\frac{\lambda_0^n\lambda_1^n}{(n-1)!n!}\xi^{n-1}(t-\xi)^n
\end{aligned},\quad n\geq1,\end{equation}
and 
\begin{equation}\label{eq:qiimpar}\begin{aligned}
q_0(x, t; 2n+1)=&\frac{\lambda_0^{n+1}\lambda_1^n}{(n!)^2}\xi^n(t-\xi)^{n} \\
q_1(x, t; 2n+1)=&\frac{\lambda_0^n\lambda_1^{n+1}}{(n!)^2}\xi^{n}(t-\xi)^n
\end{aligned},\quad n\geq0.\end{equation}
Denote
  $\theta(x,t)=\frac{1}{c_0-c_1}\rme^{-\lambda_0\xi-\lambda_1(t-\xi)}\1_{\{0<\xi<t\}}$.

Equations  \eqref{eq:pi}  have the following solution\textup{:}
\begin{equation}\label{pqj}
\begin{aligned}
p_i(x,\; t;\; 0)=&\rme^{-\lambda_it}\d(x-c_it),\\
p_i(x,\; t;\; n)=&q_i(x-j_{in},\; t;\; n)\theta(x-j_{in},\; t),\quad n\geq 1,
   \quad i=0,\; 1,
\end{aligned}
\end{equation}
where the displacements $j_{in}$ are  defined as
the sum of alternating jumps\textup{,} $j_{in}=\sum_{k=1}^nh_{i_k},$ where 
$i_k=i,$ if $k$ is odd\textup{,} and $i_k=1-i,$ if k is even.

Summing up we obtain the solution of \eqref{eq:p-complete}:
\begin{equation}
\label{eq:Distr}
\begin{aligned}
p_i(x,\; t)=&
\rme^{-\lambda_i t}\cdot\delta_0(x-c_it)\\
+&\frac{1}{c_0-c_1}\left[\lambda_i\theta(x-h_i,\ t)
I_0\left(2\frac{\sqrt{\lambda_0\lambda_1(c_0t-x+h_i)(x-h_i-c_1t)}}{c_0-c_1}\right)
\right.\\
+&\left.\sqrt{\lambda_0\lambda_1}\theta(x, t)
\left(\frac{x-c_1t}{c_0t-x}\right)^{\frac{1}{2}-i}
I_1\left(2\frac{\sqrt{\lambda_0\lambda_1(c_0t-x)(x-c_1t)}}{c_0-c_1}\right)\right],
\end{aligned}
\end{equation}%
where $I_0(z)=\sum\limits_{n=0}^\infty\dfrac{(z/2)^{2n}}{(n!)^2}$ and
$I_1(z)=I_0'(z)$ are the modified Bessel functions. 

See the proof of \eqref{eq:qipar}-\eqref{eq:Distr} in \citet{R2007}.

\section{Generalised jump-telegraph processes: moments}\label{sec:2-2}
\setcounter{equation}{0}

Using \eqref{eq:p-complete}  the equations for the expectations can be derived also.
Let $\mu_i(t):=\EE \{X_i(t)\}$ and
 $\bar l_i(\cdot):=\EE\{l_i(T; t)\}=\int_0^\infty f_{1-i}(\tau)l_i(\tau; t)\rmd\tau,\; t\geq0$.
Equations \eqref{eq:p-complete} lead  to 
\[
\mu_i(t)=\bar F_i(t)\bar l_i(t)+\int_0^t\left(\bar l_i(s)+h_i(s)+\EE\{X_{1-i}(t-s)\}\right)f_i(s)\rmd s,\quad i=0, 1.
\]
Therefore the expectations $\mu_i,\; i=0, 1$ follow the equations of  Volterra type:
\begin{equation}
\label{eq:mu1}\begin{aligned}
\mu_0(t)=&a_0(t)+\int_0^t\mu_1(t-s)f_0(s)\rmd s,\\
\mu_1(t)=&a_1(t)+\int_0^t\mu_0(t-s)f_1(s)\rmd s,
\end{aligned}\end{equation}
where
\[
a_i(t):=\bar F_i(t)\bar l_i(t)+\int_0^t(\bar l_i(s)+h_i(s))f_i(s)\rmd s,\quad i=0, 1.
\]

Integrating by parts at the latter integral we have 
\[
\int_0^t\bar l_i(s)f_i(s)\rmd s=-\bar F_i(t)\bar l_i(t)+\int_0^t\bar c_i(s)\bar F_i(s)\rmd s, 
\]
which gives the following simplification for functions $a_i$:
\begin{equation}
\label{eq:a}
a_i(t)=\int_0^t\left(
\bar F_i(s)\bar c_i(s)+f_i(s)h_i(s)
\right)\rmd s.
\end{equation}
Here we denote 
$\bar c_i(s)=\EE\{c_i(\cdot; s)\}=\int_0^\infty f_{1-i}(\tau)c_i(\tau; s)\rmd \tau,\; i=0, 1$.

Equations for variances $\sigma_i(t):=\mathrm{var}\{X_i(t)\}=\EE\{\left(X_i(t)-\mu_i(t)\right)^2\}$
can be derived similarly:
\begin{equation}
\label{eq:var}\begin{aligned}
\sigma_0(t)=&b_0(t)+\int_0^t\sigma_1(t-s)f_0(s)\rmd s,\\
\sigma_1(t)=&b_1(t)+\int_0^t\sigma_0(t-s)f_1(s)\rmd s,
\end{aligned}\end{equation}
where
\[
b_i(t):=\bar F_i(t)\left(\bar l_i(t)-\mu_i(t)\right)^2
+\int_0^t\left(\bar l_i(s)+h_i(s)+\mu_{1-i}(t-s)-\mu_i(t)\right)^2f_i(s)\rmd s,\quad i=0,1.
\]

Generalising \eqref{eq:mu1}-\eqref{eq:var}, we have the following result.


\begin{theo}\label{th:moments}
Let $g=g(x),\;-\infty<x<\infty$ be a locally bounded measurable function.
Assume that  
\begin{equation}\label{cond:int}
\int_0^\infty f_{1-i}(\tau)|g(x+l_i(\tau; t))|\rmd \tau<\infty,\quad i=0, 1.
\end{equation}
 Then the expectations
\[
u_0(x, t)=\EE\{g(x+X_0(t))\},\quad u_1(x, t)=\EE\{g(x+X_1(t))\}
\]
exist\textup{,} and they satisfy the system
\begin{align}
\label{u0}
   u_0(x, t) =& G_0(x, t)+\int_0^\infty\int_0^tu_1(x+l_0(\tau; s)+h_0(s), t-s)f_1(\tau)f_0(s)\rmd \tau\rmd s,
   \\
\label{u1}
u_1(x, t)    =&   G_1(x, t)+\int_0^\infty\int_0^tu_0(x+l_1(\tau; s)+h_1(s), t-s)f_0(\tau)f_1(s)\rmd \tau\rmd s,
\end{align}
where $G_i(x, t)=\bar F_i(t)\int_0^\infty f_{1-i}(\tau)g(x+l_i(\tau; t))\rmd \tau,\; i=0, 1$.
\end{theo}
\begin{proof}
Equations \eqref{u0}-\eqref{u1} follow by conditioning on the first velocity reversal,
see \eqref{eq:X01}.
\end{proof}

The equations for the moments $\mu_i^{(N)}(t):=\EE\left\{ X_i(t)^N\right\},\; t\geq0,\; N\geq0$
can be derived  by using Theorem \ref{th:moments} with $g(x)=x^N$, see
\eqref{u0}-\eqref{u1}.

\begin{cor} Let $N=0, 1, 2,  \ldots$

Functions $\mu_0^{(k)}(t), \mu_1^{(k)}(t),\; t\geq0, k=0, 1,\ldots N$ satisfy the equations
\begin{equation}\label{eq:mun}
\begin{aligned}
\mu_0^{(N)}(t)=&\bar F_0(t)\int_0^\infty f_1(\tau) l_0(\tau; t)^N\rmd\tau+\sum_{k=0}^N
\begin{pmatrix}
    N   \\
      k  
\end{pmatrix}
\int_0^t g_{0, N-k}(s)\mu_1^{(k)}(t-s)f_0(s)\rmd s,\\
\mu_1^{(N)}(t)=&\bar F_1(t)\int_0^\infty f_0(\tau)l_1(\tau; t)^N\rmd\tau+\sum_{k=0}^N
\begin{pmatrix}
     N   \\
      k  
\end{pmatrix}
\int_0^t g_{1, N-k}(s)\mu_0^{(k)}(t-s)f_1(s)\rmd s.
\end{aligned}\end{equation}
Here $g_{0, 0}=g_{1,0}\equiv1$ and 
\[\begin{aligned}
g_{0, m}(t)=&\int_0^\infty f_1(\tau)\left(l_0(\tau; t)+h_0(t)\right)^m\rmd\tau,\\
g_{1, m}(t)=&\int_0^\infty f_0(\tau)\left(l_1(\tau; t)+h_1(t)\right)^m\rmd\tau,
\end{aligned}\quad m\geq1.\]
\end{cor}

In general, systems \eqref{eq:mu1}, \eqref{eq:var} and \eqref{eq:mun} 
have the form of the recursive Volterra equations of the second kind:
\begin{equation}
\label{eq:moments}
\begin{aligned}
\mu_0^{(N)}(t)=&a_{0}^{(N)}(t)+\int_0^t\mu_1^{(N)}(t-s)f_0(s)\rmd s,\\
\mu_1^{(N)}(t)=&a_{1}^{(N)}(t)+\int_0^t\mu_0^{(N)}(t-s)f_1(s)\rmd s,
\end{aligned}
\end{equation}
where $a_{i}^{(N)}(t),\; i=0, 1$ are generated by the preceding moments, $\mu_{1-i}^{(k)},\; k=0,\ldots N-1:$
\begin{equation}\label{def:a}
\begin{aligned}
a_{0}^{(N)}(t):=&
\bar F_0(t)\int_0^\infty l_0(\tau; t)^Nf_1(\tau)\rmd\tau+\sum_{k=0}^{N-1}
\begin{pmatrix}
     N   \\
      k  
\end{pmatrix}
\int_0^t g_{0, N-k}(s)\mu_1^{(k)}(t-s)f_0(s)\rmd s,\\
a_{1}^{(N)}(t):=&\bar F_1(t)\int_0^\infty l_1(\tau; t)^Nf_0(\tau)\rmd\tau+\sum_{k=0}^{N-1}
\begin{pmatrix}
     N   \\
      k  
\end{pmatrix}
\int_0^t g_{1, N-k}(s)\mu_0^{(k)}(t-s)f_1(s)\rmd s.
\end{aligned}\end{equation}
Here $N\geq1$. 

System \eqref{eq:moments} possesses a unique solution, see e.g. \citet{Linz}.
Under appropriate assumptions the solution can be found explicitly.
Consider the following example. 
Let the distributions of interarrival times are exponential:
\[
f_i(t)=\lambda_i\exp(-\lambda_it),\quad t\geq0,\; i=0, 1.
\]
In this particular case system \eqref{eq:moments} is solved by
\begin{equation}
\label{eq:mu}
\boldsymbol{\mu}(t)=\boldsymbol{a}(t)+\int_0^t\left(I+\varphi(t-s)\Lambda\right)L\boldsymbol{a}(s)\rmd s,
\end{equation}
where $\varphi(t)=(1-\rme^{-2\lambda t})/(2\lambda),\; 2\lambda:=\lambda_0+\lambda_1$. 
Here we use the matrix notations
$\boldsymbol{\mu}=(\mu_0^{(N)}, \mu_1^{(N)})',$ 
$\boldsymbol{a}=(a_{0}^{(N)}, a_{1}^{(N)})'$,
 \[
 L=\begin{pmatrix}
0  & \lambda_0   \\
\lambda_1  &  0  
\end{pmatrix}\qquad \text{and}\qquad
\Lambda=\begin{pmatrix}
  -\lambda_0    &  \lambda_0  \\
  \lambda_1    &  -\lambda_1
\end{pmatrix}
.\]

To check it, notice that system \eqref{eq:moments} is equivalent to ODE with zero initial condition:
\[
\frac{\rmd \boldsymbol{\mu}}{\rmd t}=\Lambda\boldsymbol{\mu}(t)+\boldsymbol{\phi}(t),\quad \boldsymbol{\mu}(0)=0,
\]
where $\boldsymbol{\phi}=\dfrac{\rmd\boldsymbol{a} }{\rmd t}+(L-\Lambda)\boldsymbol{a}.$
We get this equation by differentiating in \eqref{eq:moments} with subsequent integration by parts.
Clearly,  the unique solution is
\begin{equation}
\label{sol:mu}
\boldsymbol{\mu}(t)=\int_0^t\rme^{(t-s)\Lambda}\boldsymbol{\phi}(s)\rmd s.
\end{equation}
Integrating by parts in \eqref{sol:mu} we obtain 
\[
\boldsymbol{\mu}(t)
=\boldsymbol{a}(t)+\int_0^t\rme^{(t-s)\Lambda}L\boldsymbol{a}(s)\rmd s.
\]
Now, the desired representation \eqref{eq:mu} follows from 
\begin{equation}\label{def:expL}
\exp\{t\Lambda\}={\rm I}+\varphi(t)\Lambda=\frac{1}{2\lambda}
\begin{pmatrix}
 \lambda_1+\lambda_0\rme^{-2\lambda t}     &   \lambda_0(1-\rme^{-2\lambda t})   \\
   \lambda_1(1-\rme^{-2\lambda t})     &   \lambda_0+\lambda_1\rme^{-2\lambda t}  
\end{pmatrix}.
\end{equation}

\section{Martingales}\label{sec:mart}
\setcounter{equation}{0}

Let $X_0=X_0(t)$ and $X_1=X_1(t)$ be (integrated) telegraph processes
defined by \eqref{def:ijtp} on the probability space $(\Omega, \mathcal{F}, \PP)$.
Let $\mu_i(t)=\EE\{X_i(t)\},\; i=0, 1$ denotes  the expectations, and 
coefficients $a_i(t),\; i=0, 1$ are defined by 
 \eqref{eq:a}.

Notice that by \eqref{eq:mu1} 
$\mu_0=\mu_1\equiv0$ if and only if
$a_0=a_1\equiv0$,
which is equivalent to the set of identities, see \eqref{eq:a},
\begin{equation}\label{eq:DoobMeyer}
\begin{aligned}
\bar F_0(t)\bar c_0(t)+h_0(t)f_0(t)\equiv0\\
\bar F_1(t)\bar c_1(t)+h_1(t)f_1(t)\equiv0
\end{aligned}\;,\qquad t\geq0.
\end{equation}

Let $\mathcal{F}_t,\; t\geq0$ be the filtration, generated by 
$\{(X_0(s),\; X_1(s))~|~s\leq t\}$.

\begin{theo}
The integrated  jump-telegraph processes $X_0$ and $X_1$ defined by \eqref{def:ijtp}
are $\mathcal{F}_t$-martingales if and only if \eqref{eq:DoobMeyer} holds.
\end{theo}

\begin{proof}
The proof can be done by computing  the  conditional expectation
$\EE\{X_i(t_2)-X_i(t_1)~|~\F_{t_1}\}$  for $0\leq t_1\leq t_2$.
Indeed, 
\[
\EE\{X_i(t_2)-X_i(t_1)~|~\F_{t_1}\}
=\EE\left\{\int_{t_1}^{t_2}\mathcal T_i(s; c_0, c_1)\rmd s+
\sum_{n=N_i(t_1)+1}^{N_i(t_2)}
h_{\ep_i(\tau_n)}(T_n)~|~\F_{t_1}\right\}
\]
\[
=\EE\left\{
\int_0^{t_2-t_1}\mathcal T_{\ep_i(t_1+s)}(t_1+s)\rmd s+
\sum_{n=1}^{N_i(t_2)-N_i(t_1)}h_{\ep_i(\tau_{n+N_i(t_1)})}(T_{n+N_i(t_1)})~|~\F_{t_1}\right\}
\]

According to the Markov property applied to the processes 
$\ep_i=\ep_i(t),\; N_i=N_i(t)$ and $\{\tau_k\}$ we have
\[\begin{aligned}
\ep_i(t_1+s)\stackrel{D}{=}&\tilde \ep_{\ep_i(t_1)}(s),&\quad
N_i(t_1+s)\stackrel{D}{=}& N_i(t_1)+\tilde N_{\ep_i(t_1)}(s),&\quad    s\geq0,\\
\tau_{n+N(t_1)}\stackrel{D}{=}&\tilde \tau_n,&\qquad
T_{n+N(t_1)}\stackrel{D}{=}&\tilde T_n,& \quad n\geq1,
\end{aligned}\]
where $\tilde \ep(s),\; \tilde N(s),\; \tilde\tau_n$ and $\tilde T_n$ 
are copies of 
$\ep(s),\; N(s),\; \tau_n$ and $T_n$ respectively,
independent of $\mathcal F_{t_1}$. Therefore,  
\begin{equation*}\label{eq:EE}
\EE\{X_i(t_2)-X_i(t_1)~|~\F_{t_1}\}
= \EE\{\tilde X_{\ep_i(t_1)}(t_2-t_1)\}.
\end{equation*}
Here $\tilde X_{\ep_i(t_1)}$ denotes the integrated jump-telegraph process, which is
initiated from the state $\ep_i(t_1)$, and is 
based on  $\tilde \ep(s),\; \tilde N(s),\; \tilde\tau_n$ and $\tilde T_n$.  
The latter expectation is equal to zero,
$\EE\{\tilde X_{\ep_i(t_1)}(t_2-t_1)\}\equiv0$, if and only if \eqref{eq:DoobMeyer} holds.\end{proof}

\begin{rem}\label{rem:signs}
Notice that if 
\eqref{eq:DoobMeyer} holds, then the  direction of jump should be opposite 
to the \textup{(}mean\textup{)} velocity value.
\end{rem}

\begin{cor}\label{cor}
If the jump-telegraph processes $X_0$ and $X_1$ defined by \eqref{def:ijtp}
are martingales\textup{,}  then 
\begin{align}\label{cond:DM}
\frac{\bar c_i(t)}{h_i(t)}<&0 \qquad\forall t\geq 0,\\
\label{cond:dens}
\int_0^\infty\frac{\bar c_i(s)}{h_i(s)}\rmd s=&\infty, \quad i=0, 1.
\end{align}

Moreover,
$X_0$ and $X_1$ are martingales, if and only if the distribution densities of interarrival times 
 satisfy the following integral relations\textup{:}
\begin{equation}\label{eq:fi}
f_i(t)=-\frac{\bar c_i(t)}{h_i(t)}\exp\left\{\int_0^t\frac{\bar c_i(s)}{h_i(s)}\rmd s\right\},\qquad i=0, 1.
\end{equation}
\end{cor}

\begin{proof}
Inequality \eqref{cond:DM}  follows directly from \eqref{eq:DoobMeyer}.
Identities \eqref{eq:DoobMeyer} are equivalent to 
\begin{equation}\label{barch}
\frac{\bar c_i(t)}{h_i(t)}=-\frac{f_i(t)}{\bar F_i(t)}\equiv (\ln \bar F_i(t))',\quad i=0, 1.
\end{equation}
Therefore
\[
\bar F_i(t)=\exp\left\{\int_0^t\frac{\bar c_i(s)}{h_i(s)}\rmd s\right\},\qquad t\geq0,\; i=0, 1.
\]
The latter equality is equivalent to  \eqref{eq:fi}. 

Notice that by definition $\lim_{t\to+\infty}\bar F_i(t)=0.$     
 Hence, condition \eqref{cond:dens} is fulfilled.
\end{proof}

In this framework various particular
cases of the  martingale distributions and the corresponding distributions of 
 interarrival times can be presented by applying Corollary \ref{cor}. 
Consider the following examples.

  \emph{Exponential distribution.}
Assume that functions $\bar c_i(t)$ and $h_i(t)$ are proportional\textup{:}
\begin{equation}\label{eq:barch}
\frac{\bar c_i(t)}{h_i(t)}\equiv-\lambda_i,\qquad \lambda_i>0,\; i=0, 1.
\end{equation}
Relations \eqref{eq:fi} mean that the integrated jump-telegraph process
is the martingale if the
 distributions of interarrival times are exponential\textup{:} 
$f_i(t)=\lambda_i\exp(-\lambda_it),$  $t>0,\; i=0, 1.$

 
Identities \eqref{eq:barch} can be written in detail as follows.
The (observable) parameters of the model,  i. e. the regimes of velocities $c_0, c_1$
and the regimes of jumps $h_0, h_1$,
satisfy the equations
\[
\lambda_{1}\int_0^\infty\rme^{-\lambda_{1}\tau}c_0(\tau, t)\rmd \tau
=-\lambda_0h_0(t),\qquad
\lambda_{0}\int_0^\infty\rme^{-\lambda_{0}\tau}c_1(\tau, t)\rmd \tau
=-\lambda_1h_1(t)
\]
with some positive constants $\lambda_0$ and $\lambda_1$. 
These equations help to compute 
the switching intensities $\lambda_0$ and $\lambda_1$  by using the (observable)  proportion between velocity  and jump values. 
On the other hand, if mean velocity regimes are given,
$\bar c_0$ and $\bar c_1$,   from these equations we can conclude that
 small jumps occur with high frequency, and big jumps are rare.
The direction of jump should be opposite to the velocity sign, see also Remark \ref{rem:signs}.

\begin{prop}\label{prop}
In the framework of  \eqref{def:ijtp} we assume that 
the Markov flow of switching times $\frak{T}=\{\tau_k\}_{k=0}^\infty$ has interarrival intervals
 $\tau_k-\tau_{k-1},\; k\geq1$
which are exponentially distributed with alternated constant 
intensities $\mu_0, \mu_1>0$. 
 Let the velocity regimes $c_i=c_i(t)$ and
 jump amplitudes $h_i=h_i(t)$ are given, and they are proportional
 as in \eqref{eq:barch},
 $c_i(t)/h_i(t)=-\lambda_i,\; i=0, 1.$ 
 
 The martingale measure for $(X_0, X_1)$  exists and it is unique.
\end{prop}

\begin{proof} 
According to the Girsanov Theorem, see \citet{R2007}, we apply Radon-Nikodym derivative of the form
\begin{equation}\label{def:RN}
\frac{\rmd \QQ}{\rmd \PP}=\mathcal{E}_t\{X^*\}=\exp\left\{\int_0^t\mathcal T_{i}(s; c_0^*, c_1^*)\rmd s\right\}
\kappa_i^*(t),
\end{equation}
where  $\kappa_i^*(t)=\prod_{k=1}^{N_i(t)}(1+h_{\ep_i(\tau_{k-1})}^*)$ 
is produced by the jump process with constant 
jump amplitudes $h_i^*=-c_i^*/\mu_i$,
and $\int_0^t\mathcal T_{i}(s; c_0^*, c_1^*)\rmd s$ is the integrated telegraph process
with constant velocities 
$c_i^*=\mu_i-\lambda_i$. 
Under the new measure $\QQ$
the underlying Markov flow takes the intensities $\lambda_i,\; i=0, 1$ 
(see Theorem 2 and Theorem 3 by \citet{R2007}).
Therefore, process $X_i(t)$ becomes the martingale. 
\end{proof}

   \emph{Erlang distribution.}
   Telegraph processes with Erlang-distributed interarrival times have been studied by
   \cite{PSZ} and  \cite{DiC2}.
 In our setting, it is easy to see that the martingale distribution 
 can be obtained  by means of
 alternated Erlang distribution for interarrival times, 
  $f_i(t)=\frac{\lambda_i^{n_i}t^{n_i-1}}{(n_i-1)!}\rme^{-\lambda_it}\1_{\{t>0\}},\;$
  $\bar F_i(t)=\sum_{k=0}^{n_i-1}(\lambda_it)^k/k!,\;$
 $ \lambda_i>0,$ $n_i\geq 1,\; i=0, 1$,
   if    the velocities and jumps follow the proportion, see \eqref{barch},
   \begin{equation}\label{erlang-proportion}
   \bar c_i(t)/h_i(t)
  =-\frac{\lambda_i^{n_i}t^{n_i-1}/(n_i-1)!}{\sum_{k=0}^{n_i-1}(\lambda_it)^k/k!}.
   \end{equation}
   
   One can  get the martingale measure by
   changing the intensities of the underlying Poisson process
   (see Proposition \ref{prop}).
   
   More precisely, let $(\Omega, \mathcal{F}, \PP)$ be given probability space.
   Consider the Poisson flow $\frak{T}=\{\tau_k\}_{k=0}^\infty$ with constant 
switching intensities $\mu_0, \mu_1>0$. 
Let $\mathcal{G}_t,\; t\geq0$ be a filtration based on this Poisson flow.

We interpret the governing  Erlang flow $\frak{T}^{(n)}(t)$    as thinned Poisson: the system accepts each $n$-th arrived signal. Let $\mathcal{F}_t$ be the 
filtration generated by $\{\frak{T}^{(n)}(s)~:~s\leq t\}$. Clearly, $\mathcal{F}_t\subset\mathcal{G}_t,\; \forall t\geq0$.

All filtrations here are assumed to satisfy the usual hypotheses, see \citet{Protter}.

Changing the measure by means of the Radon-Nikodym derivative
defined by \eqref{def:RN} we pass from intensities $\mu_0,\; \mu_1$ to 
intensities $\lambda_0,\; \lambda_1$ (defined by \eqref{erlang-proportion}) 
for the underlying  Poisson process. Therefore, under the new measure 
the telegraph process with jumps, $X=X(t),\; t\geq0$ is  $\mathcal{G}_t$-martingale. Then
$X=X(t),\; t\geq0$ is  again martingale, for the filtration $\mathcal{F}_t$, see Theorem 2.2, \citet{FP}.

   Another particular possibilities are the following.

\begin{enumerate}
   
\item \emph{Weibull distribution.}
  Assuming that
  \[
\bar c_i(t)/h_i(t)=-\lambda_it^{\alpha_i},\quad \alpha_i>-1,\; \lambda_i>0, \; i=0,1,
  \]
we have $f_i(t)=\lambda_it^{\alpha_i}\exp\left\{-\frac{\lambda_i}{\alpha_i+1}t^{\alpha_i+1}\right\}\1_{t>0}$.

      \item \emph{Pareto distribution.}  Let $0<\lambda_0, \la_1<2.$ For $b_0, b_1>0$
      assume that 
          \[
   \bar c_i(t)/h_i(t)=-\frac{\la_i}{t}\cdot\1_{\{t>b_i\}},\quad i=0, 1.
\]  
   Hence, the martingale distribution is determined by a Pareto distribution for interarrival times,
  i. e. $f_i(t)=\la_i b_i^{\la_i} t^{-1-\la_i} \1_{\{t>b_i\}},\; i=0, 1.$ 

 This distribution is in the domain of normal attraction of some $\la_i$-stable distribution, see \citet{Feller}.

   \item \emph{Logistic distribution.}
   Let interarrival times $T_n,\; n\in\ZZ$ have (alternatied) logistic distributions with the density
   $f_i(t)=\frac{2\lambda_i\rme^{-\lambda_it}}{(1+\rme^{-\lambda_it})^2}\1_{\{t\geq0\}}$,
   see 
   \cite{DiC1}. This produces the martingale distribution, if  
     \[
   \bar c_i(t)/h_i(t)=-\frac{\lambda_i\rme^{-\lambda_it}}{1+\rme^{-\lambda_it}},\quad t\geq0.
   \]
   
   \item
   \emph{Cauchy distribution.} The distribution $f_i$ takes the form of Cauchy, such that
   $f_i(t)=\frac{2a_i/\pi}{a_i^2+t^2}\1_{\{t\geq0\}}$, if 
   \[
   \bar c_i(t)/h_i(t)=-\frac{a_i}{(a_i^2+t^2)(\frac{\pi}{2}-\arctan(t/a_i))},
   \quad t\geq0.
   \]

  \item  \emph{Uniform distribution.} Let
    \[
  \bar c_i(t)/h_i(t)=-\frac{1}{A_i-t}\1_{0<t<A_i}.
  \]
  Then, in this blow-up case,
  we have the uniform distribution\textup{,} $f_i(t)=\dfrac{1}{A_i}\1_{0<t<A_i}.$
\end{enumerate}

\bibliographystyle{elsarticle-harv}

\begin{thebibliography}{99}
\bibitem[Di Crescenzo (2001)]{DiC2}
{ Di Crescenzo, A.}  2001. 
On random motions with vetocities alternating at Erlang-distributed random times. 
{Adv. Appl. Probab.} 
{33}, 690-701.

\bibitem[Di Crescenzo and Martinucci (2010)]{DiC1}
{Di Crescenzo, A., Martinucci, B.}  2010. 
A damped telegraph random process with logistic
stationary distribution. 
{J. Appl. Probab.} 
{47}, 84-96.


\bibitem[Di Crescenzo and Martinucci (2013)]{DiC2013}
{Di Crescenzo, A., Martinucci, B.} 
2013.
On the generalized telegraph process
with deterministic jumps.
{Methodol. Comput. Appl. Probab.}
{15}, Issue 1,  215-235.


\bibitem[Feller (1971)]{Feller}
Feller, W. 1971. 
{An Introduction to Probability Theory and Its
Applications}, vol.~II, 2nd edition, Wiley Series in Probability and Mathematical Statistics,
Wiley, New York, 1971.

\bibitem[F\"ollmer and Protter (2011)]{FP}
F\"ollmer, H., Protter, P. 2011. Local martingales and filtration shrinkage.
ESAIM: Probability and Statistics, May 2011, Vol. 15, 25-38.


\bibitem[Hadeler (1999)]{Hadeler-rev}
Hadeler, K. P. 1999. Reaction transport systems in biological modelling,
In Mathematics inspiring
by biology, {Lect. Notes in Math.} {1714} 95-150 (Springer, 1999).

\bibitem[Linz (1985)]{Linz}
Linz P. (1985).
{Analytical and Numerical Methods
for Volterra Equations.}
SIAM. Philadelphia. 1985

\bibitem[L\'opez and Ratanov (2012)]{ON2012}
L\'opez, O., Ratanov, N., 2012. Option pricing under jump-telegraph model with random jumps. 
J. Appl. Probab., 
{49}, 838-849.


\bibitem[Perry et al (1999)]{PSZ}
 Perry, D., Stadje, W., Zacks, S. 1999.
 First-exit times for increasing compound processes. Commun. Statist. - Stochastic Models, 
15(5), 977-992.



\bibitem[Protter (2005)]{Protter}
 Protter, P. 2005.
 Stochastic Integration and Differential Equations, 2nd edition, 
 Version 2.1. Springer-Verlag, Heidelberg.

\bibitem[Ratanov (2007)]{R2007}
Ratanov, N.  2007.
A jump telegraph model for option pricing, Quant. Finance, 7, 575-583.


\bibitem[Stadje and Zacks (2004)]{Zacks2}
Stadje, W.,  Zacks, S. 2004.
Telegraph Processes with Random Velocities
Journal of Applied Probability, Vol. 41, No. 3 (Sep., 2004), 665-678

\bibitem[Zacks(2004)]{Zacks1}
Zacks, S., 2004.
Generalized integrated telegraph processes and the distribution of related stopping times.
J. Appl. Probab. 41, 497-507.
\end{thebibliography}

\end{document}